\documentclass[12pt]{article}

\usepackage[utf8]{inputenc}
\usepackage[T1]{fontenc}
\usepackage{mathpazo}

\usepackage[english]{babel}

\usepackage[a4paper,margin=2.5cm]{geometry}
\usepackage{parskip}

\newcommand{\affiliation}{\footnote}
\newcommand{\affiliationmark}[1][\value{footnote}-1]{\footnotemark[\numexpr#1+1\relax]}
\usepackage{xcolor}
\definecolor{cblue}{RGB}{0,70,140}
\definecolor{cgreen}{RGB}{100,140,0}
\definecolor{cred}{RGB}{190,10,50}

\usepackage[shortlabels]{enumitem}
\setlist{itemsep=0ex,topsep=0ex,parsep=0.4ex}

\usepackage{amsmath,amsfonts,amssymb,amsthm,mathtools,thmtools,thm-restate,bbm,centernot}

\usepackage[hyphens]{url} 
\usepackage[hidelinks,colorlinks,pagebackref]{hyperref} 
\hypersetup{citecolor=cgreen,linkcolor=cblue,urlcolor=cblue}
\usepackage[capitalise,nameinlink,noabbrev,compress]{cleveref} 

\renewcommand*{\backref}[1]{}
\renewcommand*{\backrefalt}[4]{
	\ifcase #1 Not cited.%
	\or $\uparrow$#2%
	\else $\uparrow$#2%
	\fi%
}

\let\oldbibliography\bibliography
\renewcommand{\bibliography}[1]{
  {

    \hypersetup{linkcolor=cred}
    \bibliographystyle{bibstyle}
    \oldbibliography{#1}
  }
}

\theoremstyle{plain}
\newtheorem{theorem}{Theorem}[section]
\newtheorem{lemma}[theorem]{Lemma}

\newtheorem{corollary}[theorem]{Corollary}
\newtheorem{conjecture}[theorem]{Conjecture}

\theoremstyle{definition}

\makeatletter
\renewenvironment{proof}[1][\proofname]
{\par\pushQED{\qed}
	\normalfont\topsep6\p@\@plus6\p@\relax\trivlist
	\item[\hskip\labelsep\bfseries#1\@addpunct{.}]
	\ignorespaces}
{\popQED\endtrivlist\@endpefalse}
\makeatother

\let\emptyset\varnothing
\newcommand{\eps}{\varepsilon}

\newcommand{\pr}{\mathbb{P}}
\newcommand{\ev}{\mathbb{E}}
\newcommand{\iso}{\cong}
\DeclarePairedDelimiter{\abs}{\lvert}{\rvert}
\DeclarePairedDelimiter{\ceil}{\lceil}{\rceil}
\DeclarePairedDelimiter{\floor}{\lfloor}{\rfloor}

\newcommand{\cD}{\mathcal{D}}
\newcommand{\cO}{\mathcal{O}}
\newcommand{\cR}{\mathcal{R}}
\newcommand{\cS}{\mathcal{S}}
\newcommand{\cT}{\mathcal{T}}
\newcommand{\bR}{\mathbb{R}}
\newcommand{\bZ}{\mathbb{Z}}

\title{Anticoncentration of random spanning trees in graphs with large minimum degree}
\author{Veronica Bitonti\affiliation{Mathematical Institute, University of Oxford, United Kingdom (\textsf{\{\href{mailto:veronica.bitonti@maths.ox.ac.uk}{veronica.bitonti},\href{mailto:lukas.michel@maths.ox.ac.uk}{lukas.michel},\allowbreak\href{mailto:alexander.scott@maths.ox.ac.uk}{alexander.scott}\}\allowbreak @maths.ox.ac.uk}). Research of Veronica Bitonti is supported by EPSRC grant EP/Z534870/1. Research of Alex Scott supported by EPSRC grant EP/X013642/1.} \and Lukas Michel\affiliationmark[1] \and Alex Scott\affiliationmark[1]}

\date{18 March 2026}

\begin{document}

\maketitle

\begin{abstract}
    A classical result by Otter shows that the complete graph has an exponential number of non-isomorphic spanning trees. This was recently extended by Lee to every almost regular graph of sufficiently large degree.
    
    In this paper, we consider graphs of large minimum degree. We show that every connected graph $G$ with $n$ vertices and minimum degree $d$ has at least $n^{\Omega(d)}$ non-isomorphic spanning trees. This is tight up to the constant factor in the exponent. In fact, we prove the following anticoncentration result: if $\mathcal{T}$ is a uniformly random spanning tree of $G$, then for every tree $T$, the probability that $\mathcal{T}$ is isomorphic to $T$ is at most $n^{-\Omega(d)}$. This proves a conjecture of Lee in a strong form.
\end{abstract}

\section{Introduction}

The study of tight bounds on the number of spanning trees of a given connected graph is a fundamental problem in combinatorics. This topic dates back to the 19th century when Cayley \cite{Cayley} showed that the complete graph on $n$ vertices has exactly $n^{n-2}$ spanning trees. Even earlier, a result of Kirchhoff \cite{Kirchhoff} established that the number of spanning trees of any graph can be determined exactly from the eigenvalues of its Laplacian matrix.

Since then, there has been significant research on extending these results to obtain tight bounds for the number of spanning trees in various graph classes. For the important class of connected $d$-regular graphs with $n$ vertices, bounds of McKay \cite{McKay} and Alon \cite{Alon} showed that asymptotically the number of spanning trees is $d^{(1-o_d(1)) n}$. Kostochka \cite{Kostochka} later extended this result to graphs of large minimum degree.

\begin{theorem}[Kostochka \cite{Kostochka}]\label{th:Kostochka}
    Let $G$ be a connected graph with $n$ vertices and minimum degree $d$, and let $d(G) \coloneqq \prod_{v \in V(G)} d_G(v)$. Then, the number of spanning trees of $G$ is at least $d(G) \cdot e^{-\cO((\log d)^2/d) n}$ and at most $d(G) / (n-1)$.
\end{theorem}

Instead of considering all spanning trees, it is natural to ask what happens if we count only non-isomorphic spanning trees. For the complete graph on $n$ vertices, a celebrated result by Otter \cite{Otter} from 1948 shows that there are $(1+o(1)) C \alpha^n / n^{5/2}$ non-isomorphic spanning trees where $C \approx 0.535$ and $\alpha \approx 2.956$. Very recently, Lee \cite{Lee} also proved that every almost regular graph of sufficiently large degree has an exponential number of non-isomorphic spanning trees.

\begin{theorem}[Lee \cite{Lee}]\label{thm:lee}
    There exists $\delta > 0$ such that if $d$ is sufficiently large and $G$ is a connected graph with $n$ vertices whose degrees are in the range $[(1-\delta) d, (1+\delta) d]$, then the number of non-isomorphic spanning trees of $G$ is at least $e^{\Omega(n)}$.
\end{theorem}

Lee proved this result by establishing an anticoncentration property for random spanning trees of $G$. More precisely, Lee \cite{Lee} showed that if $\cT$ is a uniformly random spanning tree of $G$, then for every tree $T$ it holds that
\[
    \pr(\cT \iso T) \le e^{-\Omega(n)}.
\]
This anticoncentration property immediately implies \cref{thm:lee}.

While the results of Otter \cite{Otter} and Lee \cite{Lee} provide analogues for the results of Cayley \cite{Cayley} and Alon \cite{Alon} in the setting of non-isomorphic spanning trees, there is no such analogue for the result of Kostochka \cite{Kostochka}. Indeed, the complete bipartite graph $K_{d,n-d}$ has only $n^{\cO(d)}$ non-isomorphic spanning trees, and so despite having a large minimum degree, its number of non-isomorphic spanning trees does not grow exponentially in $n$. However, since $K_{d,n-d}$ still has polynomially many non-isomorphic spanning trees, Lee \cite{Lee} conjectured that at least a polynomial anticoncentration property should hold in graphs with large minimum degree.

\begin{conjecture}[Lee \cite{Lee}]\label{conj:lee}
    Let $d$ be sufficiently large, let $G$ be a connected graph with $n$ vertices and minimum degree at least $d$, and let $\cT$ be a uniformly random spanning tree of $G$. Then, for every tree $T$ it holds that
    \[
        \pr(\cT \iso T) \le n^{-\Omega(1)}.
    \]
\end{conjecture}

The methods of Lee \cite{Lee} do not yield a proof of this conjecture since they rely on counting the number of embeddings of a fixed tree in an almost regular graph and showing that this number is much less than the total number of spanning trees of the graph. While \cref{th:Kostochka} provides very good lower and upper bounds on the number of spanning trees, these bounds still differ by a factor that is exponential in $n$, and so this approach cannot prove any bounds that are only polynomial in $n$.

In this paper, we prove \cref{conj:lee} using an entirely different approach. In fact, our approach establishes a strong form of this conjecture, with the exponent growing linearly in $d$. As the example of $K_{d,n-d}$ shows, this anticoncentration property is optimal up to the constant factor in the exponent.

\begin{theorem}\label{thm:anticoncentrationminimumdegree}
    Let $d$ be sufficiently large, and let $n$ be sufficiently large relative to $d$. Suppose that $G$ is a connected graph with $n$ vertices and minimum degree at least $d$, and let $\cT$ be a uniformly random spanning tree of $G$. Then, for every tree $T$ it holds that
    \[
        \pr(\cT \iso T) \le n^{-\Omega(d)}.
    \]
\end{theorem}

In particular, it immediately follows that every graph with large minimum degree has polynomially many non-isomorphic spanning trees.

\begin{corollary}\label{thm:numbernonisospanningtreesminimumdegree}
    Let $d$ be sufficiently large, and let $n$ be sufficiently large relative to $d$. Suppose that $G$ is a connected graph with $n$ vertices and minimum degree at least $d$. Then, the number of non-isomorphic spanning trees of $G$ is at least $n^{\Omega(d)}$.
\end{corollary}

Our main strategy for proving \cref{thm:anticoncentrationminimumdegree} is to \emph{reconfigure} the leaves of a uniformly random spanning tree $\cT$ of $G$. This means that we select a subset of the leaves of $\cT$, disconnect them from their parent in $\cT$, and then reattach them to the tree by connecting them to a neighbour in $G$. We will do this carefully in such a way that the resulting tree $\cT'$ is again a uniformly random spanning tree of $G$, and so it suffices to show that $\pr(\cT' \iso T) \le n^{-\Omega(d)}$.

In fact, we will prove the stronger result that the degree sequence of $\cT'$ coincides with the degree sequence of any given $T$ with probability at most $n^{-\Omega(d)}$. In particular, this shows that $\cT'$ and $T$ are unlikely to be isomorphic. To prove that these degree sequences are likely to differ, we will show that $\cT$ has a linear number of leaves that we can reconfigure, and that each leaf that we reconfigure has many neighbours at which we can reattach the leaf to the tree. Due to these many options of reattaching the leaves to the tree, we will show that a uniformly random reconfiguration is unlikely to yield any fixed degree sequence, as required.

The rest of the paper is structured as follows. In \cref{sec:manyleaves}, we show that in a graph with large minimum degree, a uniformly random spanning tree has a linear number of leaves with high probability. We believe that this result might be of independent interest. In \cref{sec:leafreconfiguration}, we then explain how to reconfigure a linear number of these leaves. Using the anticoncentration properties of the degree sequence of the spanning tree after this reconfiguration, we then prove \cref{thm:anticoncentrationminimumdegree}. We finish in \cref{sec:openproblems} with some open problems. Throughout the paper, we assume that $n$ is sufficiently large.

\section{Random spanning trees with many leaves}\label{sec:manyleaves}

As mentioned in the introduction, our main strategy for proving the anticoncentration property for a uniformly random spanning tree is to reconfigure the leaves of such a spanning tree and to show that the resulting tree is unlikely to have any fixed degree sequence. Since we want to reconfigure many leaves in order to obtain many different degree sequences, we first show that a uniformly random spanning tree of a graph with large minimum degree has a linear number of leaves.\footnote{For almost regular graphs with sufficiently large degree, a bound of this form can be deduced from the proof of \cite[Claim 3.3]{Lee} in combination with the result of Kostochka \cite{Kostochka}.}

\begin{lemma}\label{lem:unispantreemanyleaves}
    Let $d \ge 2$, and let $n$ be sufficiently large. Suppose that $G$ is a connected graph with $n$ vertices and minimum degree at least $d$, and let $\cT$ be a uniformly random spanning tree of $G$. Then,
    \[
        \pr\left(\cT \text{ has at most } \frac{n}{8} \text{ leaves}\right) \le e^{-(1/32-\cO((\log d)^2/d)) n}.
    \]
\end{lemma}

To prove this result, we consider a digraph obtained from $G$ by choosing one random out-neighbour for each vertex of $G$. Such random digraphs were previously used by Alon \cite{Alon} and Kostochka \cite{Kostochka} to bound the number of spanning trees of $G$, and by Lee \cite{Lee} to prove anticoncentration for spanning trees with few leaves in almost regular graphs. We sample this random digraph and then pass to the underlying undirected graph. It can be shown that this yields every spanning tree with the same non-zero probability. Since the number of ways of choosing such a digraph is
\[
    d(G) \coloneqq \prod_{v \in V(G)} d_G(v),
\]
the result of Kostochka \cite{Kostochka} implies that this process is not too unlikely to yield a spanning tree. We will show that the probability of having a sublinear number of vertices with degree 1 is much more unlikely, and so a uniformly random spanning tree is likely to have many leaves. This last part of the argument will rely on McDiarmid's inequality \cite{McDiarmid}, which we will use several times throughout this paper.

\begin{lemma}[McDiarmid's inequality]\label{lem:mcdiarmidsinequality}
    Let $X_1, \dots, X_n$ be independent random variables and let $X \in \bR$ be a random variable determined by $X_1, \dots, X_n$ such that changing a single random variable $X_i$ changes the value of $X$ by at most $c \in \bR$. Then, for any $\eps > 0$ it holds that
    \[
        \pr(X \le \ev(X) - \eps) \le e^{-2 \eps^2 / (n c^2)}.
    \]
\end{lemma}

\begin{proof}[Proof of \cref{lem:unispantreemanyleaves}]
    We say that a \emph{1-out-directed graph} in $G$ is a directed graph $D$ obtained from $G$ by choosing for each vertex $v \in V(G)$ exactly one neighbour $u \in N(v)$ as the unique out-neighbour of $v$ in $D$. Clearly, the total number of 1-out-directed graphs in $G$ is $d(G)$. For a directed graph $D$, we write $U(D)$ for the underlying simple undirected graph of $D$.

    Suppose that $T$ is a spanning tree of $G$. Note that the number of 1-out-directed graphs $D$ with $U(D) = T$ is exactly $n-1$. Indeed, for every edge $e$ of $T$, we may obtain such a 1-out-directed graph by directing $e$ in both directions and directing every other edge of $T$ towards $e$. Conversely, if $U(D) = T$, then as $D$ must have a directed cycle, this cycle can only consist of a single edge $e$ of $T$ directed in both directions, and every other edge of $T$ must be directed towards $e$ in order for $D$ to be 1-out-directed. Additionally, we know by \cref{th:Kostochka} that $G$ has at least $d(G) \cdot e^{-\cO((\log d)^2 / d) n}$ spanning trees. 
    
    Let $\cD$ be a uniformly random 1-out-directed graph in $G$. Then, by the above discussion, it follows that
    \[
        \pr(U(\cD) \text{ is a tree}) \ge e^{-\cO((\log d)^2 / d) n}.
    \]
    Moreover, if we condition $\cD$ on the event that $U(\cD)$ is a tree, then $U(\cD)$ is a uniformly random spanning tree of $G$.

    We say that a vertex $v$ of $\cD$ is a \emph{leaf} if its in-degree is 0. Note that this means that $v$ is also a leaf in $U(\cD)$. For each vertex $v \in V(G)$, let $s(v) \coloneqq \sum_{u \in N(v)} 1/d(u)$. Then,
    \[
        \pr(v \text{ is a leaf of } \cD) = \prod_{u \in N(v)} \left(1 - \frac{1}{d(u)}\right) \ge \prod_{u \in N(v)} 4^{-1 / d(u)} = 4^{-s(v)},
    \]
    where the inequality used the fact that $d(u) \ge 2$ and $1 - x \ge 4^{-x}$ for $x \in [0, 1/2]$ as $4^{-x}$ is convex and equality holds for $x \in \{0, 1/2\}$. Observe that $\sum_v s(v) = n$. Therefore, Jensen's inequality implies that
    \[
        \ev(\text{number of leaves of } \cD) \ge \sum_{v \in V(G)} 4^{-s(v)} \ge n \cdot 4^{-\sum_v s(v) / n} = \frac{n}{4}.
    \]

    Observe that $\cD$ is sampled by independently choosing for each vertex $v \in V(G)$ a uniformly random neighbour of $v$ as the unique out-neighbour of $v$ in $\cD$. If we change the out-neighbour of a single vertex in $\cD$, this can create at most one leaf and eliminate at most one leaf, and so this changes the number of leaves of $\cD$ by at most one. Thus, \cref{lem:mcdiarmidsinequality} implies that
    \[
        \pr\left(\cD \text{ has at most } \frac{n}{8} \text{ leaves}\right) \le e^{-2(n/8)^2/n} = e^{-n/32}.
    \]
    In particular, it follows that
    \begin{align*}
        \pr\left(\cT \text{ has at most } \frac{n}{8} \text{ leaves}\right) & \le \pr\left(\cD \text{ has at most } \frac{n}{8} \text{ leaves} \;\middle|\; U(\cD) \text{ is a tree} \right) \\
        & \le \frac{\pr\left(\cD \text{ has at most } \frac{n}{8} \text{ leaves}\right)}{\pr(U(\cD) \text{ is a tree})} \le \frac{e^{-n/32}}{e^{-\cO((\log d)^2 / d) n}}. \qedhere
    \end{align*}
\end{proof}

The arguments in this proof could be used to show that in a graph with minimum degree at least $d$, a uniformly random spanning tree has at least $(1-o_d(1)) \cdot n/e$ leaves with high probability.\footnote{The arguments of Lee \cite{Lee} could have been used to show that this holds in regular graphs with sufficiently large degree.} Asymptotically, this lower bound matches the expected number of leaves of a uniformly random tree on $n$ vertices.\footnote{This follows from Cayley's formula. Indeed, a vertex $v$ is a leaf in $(n-1) (n-1)^{n-3}$ trees as there are $n-1$ ways to attach $v$ to any tree on the remaining $n-1$ vertices, and so the probability that $v$ is a leaf in a uniformly random tree is $(n-1)^{n-2} / n^{n-2} = (1-1/n)^{n-2} \to 1/e$ as $n \to \infty$.}

\section{Anticoncentration for random spanning trees}\label{sec:leafreconfiguration}

In this section, we prove our main result, \cref{thm:anticoncentrationminimumdegree}. We begin by describing, in general, how to reconfigure the leaves of a spanning tree. Then, we define the specific leaf reconfiguration strategy that we will apply to a uniformly random spanning tree $\cT$ of $G$. We will ensure that this strategy reconfigures a linear number of leaves of $\cT$ and that the resulting spanning tree $\cT'$ is again distributed uniformly at random. Using the former, we then show that this leaf reconfiguration results in anticoncentration for the degree sequence of $\cT'$, and so it is unlikely that $\cT'$ is isomorphic to any fixed tree. Since $\cT'$ is a uniformly random spanning tree of $G$, this completes the proof.

\subsection{Leaf reconfiguration strategies}

Let $T$ be a spanning tree of $G$ and denote the set of leaves of $T$ by $L(T)$. For each leaf $v \in L(T)$, its \emph{parent} $p_T(v)$ is its unique neighbour in $T$. A \emph{leaf selection} for $T$ is a pair $(L,P)$ where $L \subseteq L(T)$ is a set of leaves that we want to reconfigure and $P = (P(v))_{v \in L}$ are sets of \emph{potential parents} such that $p_T(v) \in P(v) \subseteq N_G(v) \setminus L$ for each leaf $v \in L$. An \emph{$(L,P)$-leaf reconfiguration} of $T$ consists of removing the edge between each leaf $v \in L$ and its parent $p_T(v)$ from $T$ and adding an edge between $v$ and some vertex $u \in P(v)$ to $T$. Note that the resulting subgraph is again a spanning tree of $G$. We say that this is a \emph{uniformly random $(L,P)$-leaf reconfiguration} if, independently for each leaf $v \in L$, we choose the new parent $u \in P(v)$ uniformly at random among all vertices of $P(v)$.

Our plan is to take a uniformly random spanning tree $\cT$ of $G$ and to apply a uniformly random leaf reconfiguration to a subset of the leaves of $\cT$. It turns out to be useful to randomise the leaf selection, as this will later enable us to find many leaves which we can reconfigure. Formally, we say that a \emph{leaf reconfiguration strategy} $\cS$ is a function that takes as input a spanning tree $T$ and the value of a random variable $\cR$ and that produces as output a leaf selection $(L,P)$ for $T$.

Later, we will choose a leaf reconfiguration strategy $\cS$ such that when we apply a uniformly random $\cS(\cT,\cR)$-leaf reconfiguration to $\cT$, the resulting spanning tree $\cT'$ is distributed uniformly at random among all spanning trees of $G$. To ensure that this is the case, call a leaf reconfiguration strategy $\cS$ \emph{reversible} if for every spanning tree $T$ and every possible value $R$ of the random variable it holds that if $T'$ can be obtained from $T$ by an $\cS(T,R)$-leaf reconfiguration, then $\cS(T',R) = \cS(T,R)$. This means that, for a fixed $R$, a reversible leaf reconfiguration strategy generates a partition of all spanning trees, and a uniformly random $\cS(T,R)$-leaf reconfiguration then resamples a spanning tree uniformly within each set of this partition. 

We claim that this property is sufficient to guarantee that $\cT'$ is distributed uniformly at random.

\begin{lemma}\label{lem:revleafreconunistatdist}
    Let $G$ be a connected graph, let $\cT$ be a uniformly random spanning tree of $G$, and let $\cS$ be a reversible leaf reconfiguration strategy whose random variable $\cR$ is independent of $\cT$. If $\cT'$ is obtained from $\cT$ by a uniformly random $\cS(\cT,\cR)$-leaf reconfiguration, then $\cT'$ is a uniformly random spanning tree of $G$.
\end{lemma}

We remark that this could be deduced from the fact that the Markov chain on spanning trees of $G$ whose transitions are induced by uniformly random $\cS(T,\cR)$-leaf reconfigurations is time-reversible. This is also the main reason why we call such strategies ``reversible''. For completeness, we provide a self-contained proof.

\begin{proof}
    Write $T \xrightarrow{(L,P)} T'$ for the event that a uniformly random $(L,P)$-leaf reconfiguration of $T$ produces the spanning tree $T'$. Consider any fixed value $R$ of the random variable and let $(L,P) \coloneqq \cS(T,R)$. Then, if $T'$ can be obtained from $T$ by an $(L,P)$-leaf reconfiguration, observe that $T$ can also be obtained from $T'$ by an $(L,P)$-leaf reconfiguration and $\cS(T',R) = (L,P)$ since $\cS$ is reversible, and so
    \[
        \pr\left(T \xrightarrow{\cS(T,R)} T'\right) = \prod_{v \in L} \frac{1}{\abs{P(v)}} = \pr\left(T' \xrightarrow{\cS(T',R)} T\right).
    \]
    Otherwise, both probabilities must be $0$. Additionally, $\pr(\cT = T) = \pr(\cT = T')$ as $\cT$ is a uniformly random spanning tree. Using the fact that $\cR$ is independent of $\cT$, it follows that
    {\allowdisplaybreaks%
    \begin{align*}
        \pr(\cT' = T') & = \sum_{T,R} \pr(\cT = T) \cdot \pr(\cR = R) \cdot \pr\left(T \xrightarrow{\cS(T,R)} T'\right) \\
        & = \sum_{T,R} \pr(\cT = T') \cdot \pr(\cR = R) \cdot \pr\left(T' \xrightarrow{\cS(T',R)} T\right) \\
        & = \pr(\cT = T') \cdot \sum_R \left(\pr(\cR = R) \cdot \sum_T \pr\left(T' \xrightarrow{\cS(T',R)} T\right)\right) \\
        & = \pr(\cT = T') \cdot \sum_R \pr(\cR = R) \\
        & = \pr(\cT = T').
    \end{align*}}%
    This shows that $\cT'$ has the same distribution as $\cT$ and is therefore a uniformly random spanning tree of $G$.
\end{proof}

\subsection{A leaf reconfiguration strategy for a random spanning tree}\label{ssec:theleafreconfigurationstrategy}

We now define the specific leaf reconfiguration strategy that allows us to reconfigure a linear number of leaves of $\cT$. By \cref{lem:unispantreemanyleaves}, we know that $\cT$ is very likely to have linearly many leaves. While we would like to reconfigure all leaves of $\cT$, this could fail if the neighbours (in $G$) of most leaves are leaves themselves. In this case, each leaf would only have a small set of potential parents where we could reattach the leaf to the tree, and as a result a uniformly random leaf reconfiguration could no longer guarantee that $\cT'$ is unlikely to have a fixed degree sequence.

Instead, we will select a random subset $\cR \subseteq V(G)$ and only reconfigure leaves within $\cR$. The sets of potential parents can then include any vertex that is not contained in $\cR$ and can therefore be much larger than before. Unfortunately, this still fails if the neighbours of almost all leaves of $\cT$ consist of the same $d$ vertices (for example if $G = K_{d,n-d}$). Indeed, the probability that all of these $d$ vertices are contained in $\cR$ does not decrease with $n$, and so there is a non-trivial probability that the sets of potential parents are again too small.

To circumvent this problem, we first try to reconfigure only leaves in $\cR$ that have a small degree in $G$. In the above case, none of the $d$ neighbours would be reconfigured, and so these neighbours could all be contained in the sets of potential parents. However, if many leaves of $\cT$ have a large degree in $G$, this could yield too few vertices for reconfiguration. In that case, we instead reconfigure leaves in $\cR$ that have a large degree in $G$. Since these leaves will always have many neighbours that are not contained in $\cR$, their sets of potential parents are always large. By adding some additional technical conditions, we can also ensure that this strategy is reversible.

Formally, we define our strategy as follows. Let $T$ be a spanning tree of $G$ and let $R \subseteq V(G)$ be a set of vertices. For any vertex $v \in V(G)$, let
\begin{align*}
    P_1(v) & \coloneqq \{u \in N_G(v) : u \notin R \text{ or } d_G(u) > n^{1/3}\}; \text{ and} \\
    P_2(v) & \coloneqq \{u \in N_G(v) : u \notin R \text{ or } \abs{N_T(u) \setminus R} \ge 2\}.
\end{align*}
We also define two sets of leaves of $T$, namely
\begin{align*}
    L_1 & \coloneqq \{ v \in L(T) \cap R : d_G(v) \le n^{1/3} \text{ and } p_T(v) \in P_1(v) \text{ and } \abs{P_1(v)} \ge d_G(v)/2 \}; \text{ and} \\
    L_2 & \coloneqq \{ v \in L(T) \cap R : d_G(v) > n^{1/3} \text{ and } p_T(v) \in P_2(v) \text{ and } \abs{P_2(v)} \ge d_G(v)/4\}.
\end{align*}
Then, we define a leaf reconfiguration strategy $\cS$ by
\[
    \cS(T,R) \coloneqq \begin{cases}
        (L_1,P_1) & \text{if } \abs{L_1} \ge n/256\text{, and} \\
        (L_2,P_2) & \text{otherwise.}
    \end{cases}
\]

From now on, $\cS$ will always refer to the leaf reconfiguration strategy described above. We begin by showing that this strategy is indeed reversible.

\begin{lemma}\label{lem:leafreconrev}
    The leaf reconfiguration strategy $\cS$ is reversible.
\end{lemma}

\begin{proof}
    Let $R \subseteq V(G)$ be a subset of the vertices of $G$. We will write $L_1(T)$, $L_2(T)$, and $P_2(T)$ for $L_1$, $L_2$, and $P_2$, respectively, calculated for a spanning tree $T$ and the set $R$. Note that $P_1$ depends on $R$ but not on $T$. 
    
    Let $T$ be a spanning tree of $G$ and let $T'$ be obtained from $T$ by an $\cS(T,R)$-leaf reconfiguration. In order to prove that $\cS$ is reversible, we need to show that $\cS(T,R) = \cS(T',R)$. Let $(L,P) \coloneqq \cS(T,R)$.

    We first prove that $P_2(T) = P_2(T')$. In fact, we claim that for all $u \in V(G)$ we have $\abs{N_T(u) \setminus R} \ge 2$ if and only if $\abs{N_{T'}(u) \setminus R} \ge 2$. Indeed, an $(L,P)$-leaf reconfiguration of $T$ only adds or removes edges that are incident to $L$. Since $L \subseteq R$, it follows that an edge between $u$ and any vertex in $N_G(u) \setminus R$ can only be added or removed during an $(L,P)$-leaf reconfiguration if $u \in L$. So, if $u \notin L$, then $\abs{N_T(u) \setminus R} = \abs{N_{T'}(u) \setminus R}$. On the other hand, if $u \in L$, then $u \in L(T)$ and $u \in L(T')$, and so $\abs{N_T(u) \setminus R} \le 1$ and $\abs{N_{T'}(u) \setminus R} \le 1$. In both cases we have $\abs{N_T(u) \setminus R} \ge 2$ if and only if $\abs{N_{T'}(u) \setminus R} \ge 2$, as claimed. In particular, $P_2(T) = P_2(T')$, and so from now on we will just write $P_2$.

    Next, we prove that $L_1(T) = L_1(T')$. Let $v \in R$ be such that $d_G(v) \le n^{1/3}$. We begin by showing that $v \in L(T)$ if and only if $v \in L(T')$. Indeed, if $v \in P(w)$ for some leaf $w \in L$, then it must hold that $P = P_2$ and $\abs{N_T(v) \setminus R} \ge 2$. By the above argument, it then also follows that $\abs{N_{T'}(v) \setminus R} \ge 2$, and so $v \notin L(T)$ and $v \notin L(T')$. On the other hand, if $v \notin P(w)$ for all $w \in L$, then $d_T(v) = d_{T'}(v)$. In both cases we have $v \in L(T)$ if and only if $v \in L(T')$.

    Moreover, if $v \in L(T)$, then we claim that $p_T(v) \in P_1(v)$ if and only if $p_{T'}(v) \in P_1(v)$. This is trivial if $p_T(v) = p_{T'}(v)$, so suppose that $p_T(v) \neq p_{T'}(v)$. Since $v \in L(T)$, the above argument implies that $v \notin P(w)$ for all $w \in L$, and so in order to have $p_T(v) \neq p_{T'}(v)$ we must have $v \in L$. In particular, $p_T(v) \in P(v)$ and $p_{T'}(v) \in P(v)$. But since $d_G(v) \le n^{1/3}$, the fact that $v \in L$ also implies that $L = L_1(T)$ and so $P = P_1$. In all cases we have $p_T(v) \in P_1(v)$ if and only if $p_{T'}(v) \in P_1(v)$, as claimed.

    So, we have shown that for all $v \in R$ with $d_G(v) \le n^{1/3}$, it holds that $v \in L(T)$ if and only if $v \in L(T')$, and if $v \in L(T)$ then $p_T(v) \in P_1(v)$ if and only if $p_{T'}(v) \in P_1(v)$. This proves that $L_1(T) = L_1(T')$, and so from now on we will just write $L_1$. In particular, $\cS(T,R) = (L_1,P_1)$ if and only if $\cS(T',R) = (L_1,P_1)$.

    Finally, suppose that $\cS(T,R) = (L_2(T),P_2)$. We will prove that $L_2(T) = L_2(T')$. Let $v \in R$. Again, we begin by showing that $v \in L(T)$ if and only if $v \in L(T')$. Similarly to above, if $v \in P(w) = P_2(w)$ for some $w \in L$, then $\abs{N_T(v) \setminus R} \ge 2$. This implies that $\abs{N_{T'}(v) \setminus R} \ge 2$ and so $v \notin L(T)$ and $v \notin L(T')$. On the other hand, if $v \notin P(w)$ for all $w \in L$, then $d_T(v) = d_{T'}(v)$. In both cases we have $v \in L(T)$ if and only if $v \in L(T')$.
    
    Moreover, if $v \in L(T)$, then we claim that $p_T(v) \in P_2(v)$ if and only if $p_{T'}(v) \in P_2(v)$. Similarly to above, this is trivial if $p_T(v) = p_{T'}(v)$, so suppose that $p_T(v) \neq p_{T'}(v)$. Since $v \in L(T)$, the argument in the preceding paragraph implies that $v \notin P(w)$ for all $w \in L$, and so in order to have $p_T(v) \neq p_{T'}(v)$ we must have $v \in L = L_2(T)$. In particular, $p_T(v) \in P_2(v)$ and $p_{T'}(v) \in P_2(v)$. In all cases we have $p_T(v) \in P_2(v)$ if and only if $p_{T'}(v) \in P_2(v)$, as claimed.

    So, we have shown that for all $v \in R$ it holds that $v \in L(T)$ if and only if $v \in L(T')$, and if $v \in L(T)$ then $p_T(v) \in P_2(v)$ if and only if $p_{T'}(v) \in P_2(v)$. This proves that $L_2(T) = L_2(T')$. In particular, $\cS(T',R) = (L_2(T),P_2)$.
\end{proof}

Next, we show that if the random subset $\cR$ includes each vertex of $G$ independently with probability $1/2$, then the leaf reconfiguration strategy $\cS$ allows us to reconfigure a linear number of leaves of $\cT$ as intended.

\begin{lemma}\label{lem:leafreconmanyleaves}
    Let $G$ be a connected graph with $n$ vertices and let $T$ be a spanning tree of $G$ with at least $n/8$ leaves. If $\cR \subseteq V(G)$ is a random subset that includes each vertex independently with probability $1/2$ and $(L,P) \coloneqq \cS(T,\cR)$, then
    \[
        \pr\left(\abs{L} \le \frac{n}{256}\right) \le e^{-\Omega(n^{1/3})}.
    \]
\end{lemma}

In this proof and later in the paper, we will use the following well-known multiplicative Chernoff bounds (see, for example, \cite{alon2016probabilistic}).

\begin{lemma}[Chernoff bound]\label{lem:chernoff}
    Let $X_1, \dots, X_n$ be independent random variables taking values in $\{0, 1\}$ and let $X$ denote their sum. If $\ev(X) \le \mu$ and $\delta \ge 0$, then
    \[
        \pr(X \ge (1 + \delta) \mu) \le e^{-\delta^2 \mu / (2+\delta)}.
    \]
    If $\ev(X) \ge \mu$ and $0 \le \delta \le 1$, then
    \[
        \pr(X \le (1 - \delta) \mu) \le e^{-\delta^2 \mu / 2}.
    \]
\end{lemma}

\begin{proof}[Proof of \cref{lem:leafreconmanyleaves}]
    If $\cS(T,\cR) = (L_1,P_1)$, then $\abs{L} = \abs{L_1} \ge n/256$. So, it remains to bound the probability of the event that $\cS(T,\cR) = (L_2,P_2)$ and $\abs{L_2} \le n/256$.

    Let $A \coloneqq \{ v \in L(T) : d_G(v) \le n^{1/3} \}$ be the set of leaves with low degree in $G$, and let $B \coloneqq \{ v \in L(T) : d_G(v) > n^{1/3} \}$ be the set of leaves with high degree in $G$. Note that $L_1 \subseteq A$ and $L_2 \subseteq B$. Since $L(T) = A \cup B$, we either have $\abs{A} \ge n/16$ or $\abs{B} \ge n/16$.

    We first consider the case where $\abs{A} \ge n/16$. For every $v \in A$, let $X_v$ be the indicator random variable for the event that $v \in L_1$, and so $\abs{L_1} = \sum_{v \in A} X_v$. Observe that $\pr(X_v = 1) \ge 1/8$ for all $v \in A$. Indeed, $v \in \cR$ with probability $1/2$, $p_T(v) \notin \cR$ with probability $1/2$, and $\abs{N_G(v) \setminus (\{p_T(v)\} \cup \cR)} \ge (d_G(v) - 1) / 2$ with probability $1/2$ by symmetry. As these events are independent and imply that $v \in L_1$, it follows that $\pr(X_v = 1) \ge 1/8$. So, $\ev(\abs{L_1}) \ge \abs{A}/8$. Moreover, changing whether a vertex $u \in V(G)$ is contained in $\cR$ can only change $X_v$ for neighbours $v \in N_G(u)$, and only if $d_G(u) \le n^{1/3}$. In particular, changing whether a vertex $u \in V(G)$ is contained in $\cR$ changes the value of $\abs{L_1}$ by at most $n^{1/3}$.

    So, if $\abs{A} \ge n/16$, then $\ev(\abs{L_1}) \ge \abs{A}/8 \ge n/128$ and by \cref{lem:mcdiarmidsinequality} it follows that
    \[
        \pr\left(\abs{L} \le \frac{n}{256}\right) \le \pr\left(\abs{L_1} \le \frac{n}{256}\right) \le e^{-2(n/256)^2/(n^{5/3})} = e^{-\Omega(n^{1/3})}.
    \]
    
    Otherwise, $\abs{B} \ge n/16$. By \cref{lem:chernoff}, we know for every leaf $v \in B$ that
    \[
        \pr\left(\abs{P_2(v)} \le \frac{d_G(v)}{4}\right) \le \pr\left(\abs{N_G(v) \setminus \cR} \le \frac{d_G(v)}{4}\right) \le e^{-d_G(v)/16} \le e^{-\Omega(n^{1/3})}.
    \]
    Similarly, if $d_T(p_T(v)) > n^{1/3}$, we know by \cref{lem:chernoff} that
    \begin{align*}
        \pr(p_T(v) \notin P_2(v)) & \le \pr(\abs{N_T(p_T(v)) \setminus \cR} \le 1) \\
        & \le \pr\left(\abs{N_T(p_T(v)) \setminus \cR} \le \frac{d_T(p_T(v))}{4}\right) \\
        & \le e^{-d_T(p_T(v))/16} \le e^{-\Omega(n^{1/3})}.
    \end{align*}
    So, by a union bound, the event that there exists a leaf $v \in B$ with $\abs{P_2(v)} \le d_G(v)/4$ or $d_T(p_T(v)) > n^{1/3}$ and $p_T(v) \notin P_2(v)$ has probability at most $e^{-\Omega(n^{1/3})}$. In particular, if $B' \coloneqq \{v \in B \cap \cR : d_T(p_T(v)) > n^{1/3} \text{ or } p_T(v) \in P_2(v)\}$, then $\pr(L_2 \neq B') \le e^{-\Omega(n^{1/3})}$.

    Finally, for every leaf $v \in B$, let $Y_v$ be the indicator random variable for the event that $v \in \cR$ and either $d_T(p_T(v)) > n^{1/3}$ or $p_T(v) \notin \cR$, and so $\abs{B'} \ge \sum_{v \in B} Y_v \eqqcolon Y$. Observe that $\pr(Y_v = 1) \ge 1/4$ for all $v \in B$, and so $\ev(Y) \ge \abs{B}/4$. Moreover, changing whether a vertex $u \in V(G)$ is contained in $\cR$ can only change $Y_v$ for neighbours $v \in N_T(u)$, and only if $d_T(u) \le n^{1/3}$. In particular, changing whether a vertex $u \in V(G)$ is contained in $\cR$ changes the value of $Y$ by at most $n^{1/3}$.
    
    So, as $\abs{B} \ge n/16$, we have $\ev(Y) \ge \abs{B}/4 \ge n/64$ and by \cref{lem:mcdiarmidsinequality} once again it follows that
    \[
        \pr\left(\abs{B'} \le \frac{n}{128}\right) \le \pr\left(Y \le \frac{n}{128}\right) \le e^{-2(n/128)^2/(n^{5/3})} = e^{-\Omega(n^{1/3})}.
    \]
    In particular,
    \[
        \pr\left(\abs{L} \le \frac{n}{256}\right) \le \pr(L_2 \neq B') + \pr\left(\abs{B'} \le \frac{n}{256}\right) \le e^{-\Omega(n^{1/3})}. \qedhere
    \]
\end{proof}

\subsection{Anticoncentration for random subgraphs of bipartite graphs}

Using the fact that the leaf reconfiguration strategy defined in \cref{ssec:theleafreconfigurationstrategy} will reconfigure a linear number of leaves of $\cT$, we want to show that the resulting tree $\cT'$ is unlikely to have any fixed degree sequence. We will prove this using the following anticoncentration result for a random subgraph of a bipartite graph, namely for a subgraph that is obtained by keeping one uniformly random edge incident to each vertex in one part of the bipartition. By applying this result in \cref{ssec:anticoncentration} to the bipartite graph between the leaves that we reconfigure and their potential parents, we can then prove the desired anticoncentration result for the degree sequence of $\cT'$.

\begin{lemma}\label{lem:degreeseqanticoncentration}
    Let $c \le 1$, let $d \ge 66 / c$, and let $n$ be sufficiently large relative to $d$. Suppose that $G$ is a bipartite graph with $n$ vertices and bipartition $V(G) = A \cup B$ where each vertex in $A$ has degree at least $d$ and $\abs{A} = c n$. Suppose we are given integers $a_k$ for $k \in \bZ$ and $b_v$ for $v \in V(G)$. Let $H$ be the random subgraph of $G$ obtained by independently keeping for each vertex $v \in A$ one uniformly random edge incident to $v$, and let $n_H(k)$ denote the number of vertices $v \in V(H)$ with $d_H(v) + b_v = k$. Then,
    \[
        \pr(n_H(k) = a_k \text{ for each } k) \le n^{-\Omega(c d)}.
    \]
\end{lemma}

We remark that when $G$ is an almost regular bipartite graph, Lee \cite{Lee} proved a concentration result regarding the number of vertices of degree $k$ in $H$. In contrast, \cref{lem:degreeseqanticoncentration} proves an anticoncentration result for the degree sequence of $H$ if the minimum degree in $A$ is large.

The proof of this lemma is quite technical. Roughly speaking, we divide the proof into two cases. First, if enough vertices in $B$ have a very large degree, say degree at least $\Omega(n^{1/2})$, then for each of these vertices we expect that its degree in $H$ has a very strong anticoncentration. We will show that this anticoncentration is enough to prove the conclusion of the lemma in this first case. Otherwise, there are polynomially many vertices in $B$ whose degree is at least $\Omega(c d)$. Although individually their degrees in $H$ only have a weak anticoncentration, the fact that there are many of these vertices means that $n_H(k)$ has a strong anticoncentration. This is enough to prove the conclusion of the lemma in this second case.

In this proof, we will use the well-known inequality $\binom{n}{k} / 2^n \le \sqrt{2/(\pi n)}$ that holds for all $n \ge 1$ and $0 \le k \le n$ as well as the majorisation inequality \cite{HardyLittlewoodPolya,Karamata}.

\begin{lemma}[Majorisation inequality]\label{lem:majorisation}
    Let $f: \bR \to \bR$ be convex and let $x_1 \ge \dots \ge x_n$ and $y_1 \ge \dots \ge y_n$ be such that $\sum_{i=1}^n x_i = \sum_{i=1}^n y_i$ and $\sum_{i=1}^k x_i \ge \sum_{i=1}^k y_i$ for all $k \in [n]$. Then,
    \[
        \sum_{i=1}^n f(x_i) \ge \sum_{i=1}^n f(y_i).
    \]
\end{lemma}

\begin{proof}[Proof of \cref{lem:degreeseqanticoncentration}]
    Note that we may construct $H$ from $G$ by repeatedly choosing a vertex $v \in A$ with at least two incident edges and deleting a uniformly random edge incident to $v$. If we can show that the resulting graph $G'$ after some number of these operations satisfies the conclusion of the lemma, then $G$ also satisfies this conclusion. In particular, we may and will frequently reduce the neighbourhood of some vertices $v \in A$ to a uniformly random subset of $N_G(v)$ of a fixed size. If we reduce the neighbourhood to a single vertex, this means that we reveal the neighbour of $v$ in $H$.

    Since we may reduce the neighbourhood of each vertex $v \in A$ to a set of size $d$, we will assume that every vertex in $A$ has degree exactly $d$. Let $B_1 \coloneqq \{v \in B : d_G(v) \ge c d / 2\}$. Note that $e(A, B_1) \ge d \abs{A} - (c d / 2) \abs{B \setminus B_1} \ge c d n / 2$.

    \textbf{Case 1:} Suppose that $\abs{B_1} \le n^{1/4}$. In this case, we claim that we can find at least $d/8$ pairs of vertices $(v_1, u_1), \dots, (v_t, u_t)$ from $B_1$ such that all these vertices are distinct and $\abs{N_G(v_i) \cap N_G(u_i)} \ge c d^2 n^{1/2} / 32$ for all $i \in [t]$.
    
    Indeed, suppose that we have constructed $(v_1, u_1), \dots, (v_t, u_t)$ for some $t < d/8$, and let $B_2 \coloneqq B_1 \setminus \{v_1, u_1, \dots, v_t, u_t\}$. Then,
    \[
        e(A, B_2) \ge e(A, B_1) - 2 t \abs{A} \ge \frac{c d n}{2} - \frac{c d n}{4} \ge \frac{c d n}{4}.
    \]
    By Jensen's inequality, this implies that
    \[
        \sum_{v, u \in B_2,v \neq u} \abs{N_G(v) \cap N_G(u)} = \sum_{v \in A} \binom{\abs{N_G(v) \cap B_2}}{2} \ge \abs{A} \binom{e(A, B_2) / \abs{A}}{2} \ge c n \binom{d / 4}{2}.
    \]
    In particular, there exists a pair of distinct vertices $(v_{t+1}, u_{t+1})$ in $B_2$ with
    \[
        \abs{N_G(v_{t+1}) \cap N_G(u_{t+1})} \ge \frac{c n \binom{d / 4}{2}}{\binom{\abs{B_2}}{2}} \ge \frac{c d^2 n^{1/2}}{32},
    \]
    where we used the fact that $d \ge 8$. So, we have constructed the pair $(v_{t+1},u_{t+1})$. By iterating this argument, it follows that we can construct $t$ such pairs for some $d/8 \le t < d/8 + 1$, which proves the claim.

    Now, reduce the neighbourhood of each vertex $v \in A$ to two vertices, and if the neighbourhood afterward is not one of the pairs $(v_i, u_i)$, then reduce the neighbourhood to a single vertex. Let $G'$ be the resulting graph. Note that $\ev(\abs{N_{G'}(v_i) \cap N_{G'}(u_i)}) \ge \abs{N_G(v_i) \cap N_G(u_i)} / \binom{d}{2} \ge c n^{1/2} / 16$. By \cref{lem:chernoff}, we know that
    \[
        \pr\left(\abs{N_{G'}(v_i) \cap N_{G'}(u_i)} \le \frac{c n^{1/2}}{32}\right) \le e^{-c n^{1/2} / 128}.
    \]
    That is, with an error probability of at most $e^{-\Omega(n^{1/2})}$, we may assume that $\abs{N_{G'}(v_i) \cap N_{G'}(u_i)} \ge c n^{1/2}/32 \eqqcolon s$ for all $i \in [t]$.

    Finally, reduce the neighbourhood of each vertex $v \in A$ that has two neighbours to a single vertex. Observe that the degree of each vertex $v \in V(G) \setminus \{v_1, u_1, \dots, v_t, u_t\}$ was already determined before this final step, so for these vertices we already knew the value of $d_H(v) + b_v$. In particular, if we want that $n_H(k) = a_k$ for all $k$, then there is a set $D$ of at most $2 t$ integers such that after this final step it must hold that $d_H(v_i) + b_{v_i} \in D$ for all $i \in [t]$. However, in this final step, the number of additional edges that will be removed from $v_i$ has a binomial distribution with $\abs{N_{G'}(v_i) \cap N_{G'}(u_i)} \ge s$ independent Bernoulli trials that each have probability $1/2$, and these events are independent for different vertices $v_i$. So, it follows that
    \[
        \pr(d_H(v_i) + b_{v_i} \in D \text{ for each } i \in [t]) \le \left(2 t \frac{\binom{s}{\floor{s/2}}}{2^s}\right)^t \le (2 t)^t \left(\frac{1}{\sqrt{s}}\right)^t \le d^d (c^2 n)^{-d/32}.
    \]
    Since $n$ is sufficiently large relative to $d$, this shows that
    \[
        \pr(n_H(k) = a_k \text{ for each } k) \le n^{-\Omega(d)}.
    \]

    \textbf{Case 2:} Suppose that $\abs{B_1} \ge n^{1/4}$. For each vertex $v \in B_1$, let $N_v \subseteq N_G(v)$ be a subset of size $\ceil{c d / 2} \eqqcolon s$. Since each vertex in $A$ has $d$ neighbours in $B$, there is a subset $B_2 \subseteq B_1$ of size at least $\abs{B_1} / (s d) \ge n^{1/4} / d^2$ such that the sets $N_v$ for $v \in B_2$ are disjoint.

    Then, for each vertex $v \in B_2$, reduce the neighbourhood of each vertex $u \in N_v$ to two vertices, and if these neighbourhoods afterward do not all contain $v$ or are not disjoint apart from $v$, then reduce these neighbourhoods to a single vertex. Furthermore, reduce the neighbourhood of every other vertex $u \in A \setminus (\bigcup_{v \in B_2} N_v)$ to a single vertex. Let $G'$ be the resulting graph and let $B_3 \subseteq B_2$ be the subset of those vertices $v \in B_2$ with $d_{G'}(u) = 2$ for all $u \in N_v$. Note that for each $v \in B_2$,
    \[
        \pr(v \in B_3) \ge \frac{\prod_{i=1}^s (d-i)}{\binom{d}{2}^s} \ge \frac{1}{d^{2s}},
    \]
    and these events are independent for different vertices $v \in B_2$. Therefore, $\ev(\abs{B_3}) \ge \abs{B_2} / d^{2s} \ge n^{1/4} / d^{2s+2}$ and, by \cref{lem:chernoff}, we know that
    \[
        \pr\left(\abs{B_3} \le \frac{n^{1/4}}{2 \cdot d^{2s+2}}\right) \le e^{-n^{1/4}/(8 \cdot d^{2s+2})}.
    \]
    That is, with an error probability of at most $e^{-\Omega(n^{1/4} / d^{\Theta(cd)})}$, we may assume that $\abs{B_3} \ge n^{1/4} / (2 \cdot d^{2s+2})$.

    Now, for each vertex $v \in B_3$, choose one neighbour $u_v \in N_v$, let $w_v \in N_{G'}(u_v) \setminus \{v\}$ be the other neighbour of $u_v$, and choose an order of the other neighbours $N_v \setminus \{u_v\}$. Reduce the neighbourhood of each vertex $u \in N_v \setminus \{u_v\}$ to a single vertex. If the set of vertices $u \in N_v \setminus \{u_v\}$ whose remaining neighbour is $v$ is not a prefix of the order of $N_v \setminus \{u_v\}$, then also reduce the neighbourhood of $u_v$ to a single vertex. Otherwise, if the new degree of $v$ plus $b_v$ is equal to the new degree of $w_v$ plus $b_{w_v}$, then also reduce the neighbourhood of $u_v$ to a single vertex. Let $G''$ be the resulting graph and let $B_4 \subseteq B_3$ be the subset of those vertices $v \in B_3$ with $d_{G''}(u_v) = 2$. Note that for each $v \in B_3$,
    \[
        \pr(v \in B_4) \ge \frac{s - 1}{2^{s-1}},
    \]
    and these events are independent for different vertices $v \in B_3$. Therefore, $\ev(\abs{B_4}) \ge (s-1) \abs{B_3} / 2^{s-1} \ge n^{1/4} / (2^s d^{2s+2})$ and, by \cref{lem:chernoff}, we know that
    \[
        \pr\left(\abs{B_4} \le \frac{s-1}{2^s} \abs{B_3}\right) \le e^{-(s-1) \abs{B_3}/2^{s+2}} \le e^{-n^{1/4}/(2^{s+3} d^{2s+2})}.
    \]
    That is, with an error probability of at most $e^{-\Omega(n^{1/4} / d^{\Theta(cd)})}$, we may assume that $\abs{B_4} \ge (s-1) \abs{B_3} / 2^s$. Moreover, for any $k \in \bZ$, let $B_4(k) \coloneqq \{v \in B_4 : d_{G''}(v) + b_v = k\}$. Note that for each $v \in B_3$,
    \[
        \pr(v \in B_4(k)) \le \frac{1}{2^{s-1}},
    \]
    and these events are independent for different vertices $v \in B_3$. Therefore, $\ev(\abs{B_4(k)}) \le \abs{B_3} / 2^{s-1}$ and, by \cref{lem:chernoff}, we know that
    \[
        \pr\left(\abs{B_4(k)} \ge \frac{\abs{B_3}}{2^{s-2}}\right) \le e^{-\abs{B_3}/(3 \cdot 2^{s-1})} \le e^{-n^{1/4}/(2^{s+2} d^{2s+2})}.
    \]
    Also note that if $\abs{k - b_v} \ge n$ for all $v \in B_3$, then $B_4(k) = \emptyset$. Therefore, by a union bound over at most $2 n^2$ values of $k$, it follows that with an error probability of at most $e^{-\Omega(n^{1/4} / d^{\Theta(cd)})}$, we may assume that $\abs{B_4(k)} \le \abs{B_3} / 2^{s-2}$ for each $k \in \bZ$.

    Without loss of generality, we may assume that for at least half of the vertices $v \in B_4$ we have $d_{G''}(v) + b_v < d_{G''}(w_v) + b_{w_v}$. Let $B_5 \subseteq B_4$ be the subset of these vertices. For each vertex $v \in B_4 \setminus B_5$, reduce the neighbourhood of $u_v$ to a single vertex, and let $G'''$ be the resulting graph.

    Finally, to obtain $H$, iterate through $k \in \bZ$ in increasing order. Then, for a fixed $k$, let $B_5(k+1) \coloneqq B_5 \cap B_4(k+1) = \{v \in B_5 : d_{G'''}(v) + b_v = k + 1\}$, and for each vertex $v \in B_5(k+1)$ reduce the neighbourhood of $u_v$ to a single vertex. Observe that the degree of each vertex $v \in V(G) \setminus (B_5 \cup \{w_v : v \in B_5\})$ was already determined before this final step, so for these vertices we already knew $d_H(v) + b_v$. Moreover, for those vertices $v \in B_5$ for which we had not yet reduced the neighbourhood of $u_v$ to a single vertex, it did hold that $d_{G'''}(w_v) + b_{w_v} > d_{G'''}(v) + b_v \ge k + 1$. So, the vertices $v \in B_5(k+1)$ are the only remaining vertices of $G$ that could still satisfy $d_H(v) + b_v = k$. In particular, if we want that $n_H(k) = a_k$, then we know the exact number of vertices $v \in B_5(k+1)$ for which the edge between $v$ and $u_v$ has to be removed in this final step. However, the number of vertices $v \in B_5(k+1)$ for which the edge between $v$ and $u_v$ is removed has a binomial distribution with $\abs{B_5(k+1)} \eqqcolon s_k$ independent trials that each have probability $1/2$, so the probability that this is the correct number of vertices is at most $\binom{s_k}{\floor{s_k/2}} / 2^{s_k}$. It follows that
    \[
        \pr(n_H(k) = a_k \text{ for each } k) \le \prod_{k \in \bZ} \frac{\binom{s_k}{\floor{s_k/2}}}{2^{s_k}} \le \prod_{k \in K} \sqrt{\frac{2}{\pi s_k}},
    \]
    where $K \coloneqq \{k \in \bZ : s_k \neq 0\}$. This product is maximised if its logarithm is maximised, namely
    \[
        \log\left(\prod_{k \in K} \sqrt{\frac{2}{\pi s_k}}\right) = \frac{1}{2}\sum_{k \in K} \left(\log\left(\frac{2}{\pi}\right) - \log(s_k)\right)
    \]
    Note that $\log(2/\pi) - \log(s_k)$ is a convex function. We also know that $\sum_{k \in K} s_k = \abs{B_5}$ and
    \[
        s_k = \abs{B_5(k+1)} \le \abs{B_4(k+1)} \le \frac{\abs{B_3}}{2^{s-2}} \le \frac{4}{s-1} \abs{B_4} \le \frac{8}{s-1} \abs{B_5}.
    \]
    Therefore, by \cref{lem:majorisation}, the above sum and therefore also the product is maximised if for all $k \in K$ except possibly one it holds that $s_k \in \{1, 8 \abs{B_5} / (s-1)\}$. If $\abs{K} \le \abs{B_5} / 2$, then there can be at least $\floor{(\abs{B_5} / 2) / (8 \abs{B_5} / (s-1))} = \floor{(s-1)/16}$ indices $k \in K$ with $s_k = 8 \abs{B_5} / (s-1)$ and so
    \[
        \pr(n_H(k) = a_k \text{ for each } k) \le \left(\sqrt{\frac{(s-1)}{8 \abs{B_5}}}\right)^{\floor{(s-1)/16}} \le \left(\frac{s 2^s d^{2s+2}}{n^{1/4}}\right)^{\floor{(s-1)/32}} = n^{-\Omega(cd)},
    \]
    where we used the fact that $s \ge c d / 2 \ge 33$ and that $n$ is sufficiently large relative to $d$. Otherwise, $\abs{K} \ge \abs{B_5} / 2$ and so
    \[
        \pr(n_H(k) = a_k \text{ for each } k) \le \left(\sqrt{\frac{2}{\pi}}\right)^{\abs{K}} \le \left(\frac{2}{\pi}\right)^{n^{1/4}/(2^{s+3} d^{2s+2})} = e^{-\Omega(n^{1/4} / d^{\Theta(cd)})}. \qedhere
    \]
\end{proof}

\subsection{Anticoncentration after the leaf reconfiguration}\label{ssec:anticoncentration}

We can now prove our main anticoncentration result by applying the leaf reconfiguration strategy from \cref{ssec:theleafreconfigurationstrategy} to a uniformly random spanning tree $\cT$ of $G$ and then using \cref{lem:degreeseqanticoncentration} to show that the degree sequence of the resulting spanning tree $\cT'$ has strong anticoncentration properties.

\begin{proof}[Proof of \cref{thm:anticoncentrationminimumdegree}]
    Let $d$ be sufficiently large, and let $n$ be sufficiently large relative to $d$. Suppose that $G$ is a connected graph with $n$ vertices and minimum degree $d$. Let $\cT$ be a uniformly random spanning tree of $G$, and let $T$ be a tree. We want to show that
    \[
        \pr(\cT \iso T) \le n^{-\Omega(d)}.
    \]

    Let $\cT'$ be obtained from $\cT$ by a uniformly random $\cS(\cT,\cR)$-leaf reconfiguration where $\cR \subseteq V(G)$ is a random subset that includes each vertex independently with probability $1/2$ and that is independent of $\cT$. Since $\cS$ is reversible by \cref{lem:leafreconrev}, we know by \cref{lem:revleafreconunistatdist} that $\cT'$ is a uniformly random spanning tree of $G$. Therefore, it suffices to prove that $\pr(\cT' \iso T) \le n^{-\Omega(d)}$.

    By \cref{lem:unispantreemanyleaves}, we know that $\pr(\abs{L(\cT)} \le n/8) \le e^{-\Omega(n)}$. Let $(L,P) \coloneqq \cS(\cT,\cR)$, and note that \cref{lem:leafreconmanyleaves} implies that
    \[
        \pr(\abs{L} \le n/256 \mid \abs{L(\cT)} \ge n/8) \le e^{-\Omega(n^{1/3})}.
    \]
    So,
    \begin{align*}
        \pr(\abs{L} \le n/256) & \le \pr(\abs{L(\cT)} \le n/8) + \pr(\abs{L} \le n/256 \mid \abs{L(\cT)} \ge n/8) \\
        & \le e^{-\Omega(n)} + e^{-\Omega(n^{1/3})} \le e^{-\Omega(n^{1/3})}.
    \end{align*}
    Recall that by the definition of $\cS$, for all $v \in L$ we have $\abs{P(v)} \ge d_G(v)/4 \ge d/4$.

    Let $G'$ be the bipartite graph between $A \coloneqq L$ and $B \coloneqq V(G) \setminus L$ with $N_{G'}(v) \coloneqq P(v)$ for all $v \in A$. Note that $\cT'$ is formed by taking the subgraph $\cT[B]$ of $\cT$ induced on $B$ and adding the random subgraph $H$ of $G'$ obtained by independently keeping for each vertex $v \in A$ one uniformly random edge incident to $v$. Let $n_{\cT'}(k)$ be the number of vertices $v \in V(\cT')$ with $d_{\cT'}(v) = k$. Note that $n_{\cT'}(k) = n_H(k)$ where $n_H(k)$ is the number of vertices $v \in V(H)$ with $d_H(v) + b_v = k$ where $b_v \coloneqq d_{\cT[B]}(v)$ if $v \in B$ and $b_v \coloneqq 0$ otherwise. Also, let $a_k$ be the number of vertices $v \in V(T)$ with $d_T(v) = k$. Clearly, if $\cT' \iso T$, then we must have $n_H(k) = n_{\cT'}(k) = a_k$ for all $k$. However, by \cref{lem:degreeseqanticoncentration} we know that $\pr(n_H(k) = a_k \text{ for each } k \mid \abs{L} \ge n/256) \le n^{-\Omega(d)}$. Therefore, it follows that
    \begin{align*}
        \pr(\cT' \iso T) & \le \pr(\abs{L} \le n/256) + \pr(n_H(k) = a_k \text{ for each } k \mid \abs{L} \ge n/256) \\
        & \le e^{-\Omega(n^{1/3})} + n^{-\Omega(d)} = n^{-\Omega(d)}. \qedhere
    \end{align*}
\end{proof}

\section{Open problems}\label{sec:openproblems}

In this paper, we proved an anticoncentration property for a uniformly random spanning tree in a graph with large minimum degree. Although this anticoncentration is optimal up to the constant factor in the exponent, we conjecture that the optimal constant factor in the exponent should be $1/2$. This would match the anticoncentration of a uniformly random spanning tree of $K_{d,n-d}$. Indeed, up to isomorphism, a spanning tree of $K_{d,n-d}$ is determined by a minimal subtree that connects the $d$ vertices in the small class and the number of leaves attached to each of these $d$ vertices. There are only a bounded number of configurations for the former, and in a uniformly random spanning tree, the latter has a multinomial distribution with $n - \cO(d)$ trials where each outcome has a success probability of $1/d$. Therefore, some isomorphism class is obtained with probability at least $\Omega(n^{-(d-1)/2})$.

\begin{conjecture}\label{conj:anticoncentration}
    Let $d$ be sufficiently large, and let $n$ be sufficiently large relative to $d$. Suppose that $G$ is a connected graph with $n$ vertices and minimum degree at least $d$, and let $\cT$ be a uniformly random spanning tree of $G$. Then, for every tree $T$ it holds that
    \[
        \pr(\cT \iso T) \le n^{-(1/2-o_n(1)) (d-1)}.
    \]
\end{conjecture}

What about the number of non-isomorphic spanning trees? The anticoncentration property from \cref{conj:anticoncentration} would imply that $G$ has at least $n^{(1/2-o_n(1)) (d-1)}$ non-isomorphic spanning trees. However, $K_{d,n-d}$ has $\Omega(n^{d-1})$ non-isomorphic spanning trees, and we conjecture that this larger quantity gives an essentially optimal lower bound.

\begin{conjecture}
    Let $d$ be sufficiently large, and let $n$ be sufficiently large relative to $d$. Suppose that $G$ is a connected graph with $n$ vertices and minimum degree at least $d$. Then, the number of non-isomorphic spanning trees of $G$ is at least $\Omega(n^{d-1})$.
\end{conjecture}

\bibliography{bib}

\end{document}